\documentclass[12pt,a4paper]{article}

\usepackage[utf8]{inputenc}
\usepackage[T1]{fontenc}
\usepackage{amsmath, amsthm, amssymb}
\usepackage{mathtools}
\usepackage{tikz-cd}
\usepackage{graphicx}
\usepackage{hyperref}
\usepackage{cleveref}
\usepackage{geometry}
\usepackage{enumitem}
\newtheorem{definition}{Definition}[section]
\newtheorem{theorem}{Theorem}[section]
\newtheorem{lemma}{Lemma}[section]

\newtheorem{remark}{Remark}

\title{A Natural Homomorphism between the Model Constructions of the Completeness and Compactness Theorems}
\author{Joaquim Reizi Barreto}
\date{\today}

\begin{document}
\maketitle
\tableofcontents

\bigskip
\begin{abstract}
We present a mathematically rigorous and constructive framework that relates two canonical model constructions in classical first-order logic: one via the Henkin (completeness) approach and one via a compactness-based approach. Concretely, we define two functors 
\[
F,\,G : \mathbf{th} \,\longrightarrow\, \mathbf{mod},
\]
where \(\mathbf{th}\) is the category of consistent first-order theories (over a fixed countable language) and \(\mathbf{mod}\) is the category of models with elementary embeddings as morphisms. 

The functor \(F\) is obtained by extending any theory \(t\) to a maximal consistent theory \(t^*\) via a globally fixed Henkin expansion scheme (with a predetermined enumeration and systematic introduction of witness constants from a universal set \(H\) of Henkin constants for all formulas). The functor \(G\) is constructed using a compactness-based approach: for each theory \(t\), we take \(G(t)\) to be the substructure of a fixed ultraproduct or saturated model that is generated by the interpretations of the Henkin constants.

We show there exists a natural transformation 
\[
\eta: F \,\longrightarrow\, G
\]
where each component \(\eta_t: F(t) \to G(t)\) is a natural homomorphism. Under appropriate conditions (e.g., when restricting to \(\aleph_0\)-categorical complete theories or atomic complete theories), \(\eta_t\) can be shown to be an isomorphism. However, the general claim of isomorphism for arbitrary consistent theories requires careful consideration of the construction details.

This perspective provides structural insight into how proof-theoretic and model-theoretic techniques interrelate, and may suggest applications in automated theorem proving and formal verification.
\end{abstract}

\newpage
\section{Introduction}
\label{sec:intro}

The completeness and compactness theorems have long been central to classical first-order logic. Traditionally, the completeness theorem is proved via the Henkin construction—extending a consistent theory to a maximal consistent theory using witness constants, yielding a term model—while the compactness theorem is typically established through finite satisfiability arguments or by constructing models via ultraproducts or saturation methods.

Recent advances in category theory provide a natural language to relate these distinct model constructions within a unified framework. In this paper, we consider the category \(\mathbf{th}\) of consistent first-order theories (over a fixed countable language with a predetermined signature) and the category \(\mathbf{mod}\) of models with elementary embeddings as morphisms. We define two functors 
\[
F,\,G : \mathbf{th} \to \mathbf{mod},
\]
where \(F\) represents a Henkin construction based on a globally fixed scheme of Henkin (witness) constants, and \(G\) represents a compactness-based construction. Our primary objective is to establish the existence of a natural transformation
\[
\eta : F \to G,
\]
such that for every theory \(t \in \mathbf{th}\), the component \(\eta_t: F(t) \to G(t)\) is a natural homomorphism in \(\mathbf{mod}\).

\textbf{Important clarification:} The general claim that \(\eta_t\) is an isomorphism for \emph{all} consistent theories requires careful attention to the construction details. In particular:
\begin{itemize}
\item When \(G(t)\) is defined as an arbitrary ultraproduct or saturated model (without restriction), the map \(\eta_t\) may fail to be surjective. For instance, if \(G(t)\) is a non-standard ultrapower of the Henkin model \(F(t)\), then \(G(t)\) may contain elements not in the image of \(\eta_t\).
\item To obtain an isomorphism, we must restrict \(G(t)\) to the Skolem closure (the substructure generated by the interpretations of the Henkin constants) within the ultraproduct or saturated model.
\item Alternatively, isomorphism holds under special conditions such as \(\aleph_0\)-categoricity or atomicity of the complete theory, where models are unique up to isomorphism in appropriate cardinalities.
\end{itemize}

This framework enriches our structural understanding of the interplay between proof theory and model theory and may open avenues for applications in automated theorem proving and formal verification.

\paragraph{Organization of the paper.}
Section~\ref{sec:prelim} reviews the necessary background in category theory and model theory, including precise definitions of categories, functors, and natural transformations, as well as detailed expositions of the Henkin and compactness-based model constructions with the globally fixed Henkin constant scheme. 
Section~\ref{sec:mainthm} states our main theorem and discusses its scope and limitations. 
Section~\ref{sec:mainproof} provides a constructive proof outline, relying on key lemmas. 
Section~\ref{sec:applications} explores potential applications and extensions of our results, 
and finally Section~\ref{sec:conclusion} concludes with a summary and open questions.

\section{Preliminaries}
\label{sec:prelim}
\begin{remark}[Set-Theoretic Framework]
\label{rmk:set-theory}
Throughout this paper, we work in the usual ZFC set theory. When dealing with classes that are too large to be sets (e.g., the class of all consistent $\mathcal{L}$-theories or the class of all $\mathcal{L}$-structures), we treat them as large categories in a standard manner. If desired, one may use a Grothendieck universe or a similar foundational device to manage size issues more rigorously. However, none of these set-theoretic subtleties affect the main arguments and results.
\end{remark}

\subsection{Basic Definitions and Set-Theoretic Scope}

In this paper, we work under the usual ZFC (Zermelo--Fraenkel with Choice) framework or an equivalent set theory. We note that some of the collections we deal with (e.g.\ the class of all first-order theories over a fixed language, or the class of all models) may form a proper class rather than a set. For simplicity, we often refer to these as “categories” even though strictly speaking they may be \emph{large categories} (or even $2$-categories) in the sense of category theory. This does not affect the main arguments but is worth keeping in mind for foundational rigor.
\begin{remark}[Large Categories and Foundational Rigor]
\label{rmk:large-cat}
In this paper, we refer to certain classes (e.g.\ the class of all consistent theories over a countable language, or the class of all $\mathcal{L}$-structures) as “categories” even though they may in fact be \emph{large} (proper) classes in ZFC. From a foundational viewpoint, one may regard them as $($possibly$)$ large categories or even $2$-categories within an appropriate Grothendieck universe, or treat them at the metatheoretical level in ZFC without further complications. This approach does not affect the validity of our main results but is worth noting for strict foundational completeness.
\end{remark}

A \textbf{category} is one of the most fundamental structures in mathematics that abstracts collections of objects and the morphisms between them. The following is the standard definition.

\begin{definition}\label{def:category}
A \textbf{category} $\mathcal{C}$ consists of:
\begin{enumerate}
    \item A collection $\operatorname{Ob}(\mathcal{C})$ of \emph{objects}.
    \item For every pair $A,B \in \operatorname{Ob}(\mathcal{C})$, a (possibly large) set $\operatorname{Hom}_{\mathcal{C}}(A,B)$ of \emph{morphisms}.
    \item A composition law 
    \[
      \circ \;:\; \operatorname{Hom}_{\mathcal{C}}(B,C)\;\times\;\operatorname{Hom}_{\mathcal{C}}(A,B)
         \;\longrightarrow\;\operatorname{Hom}_{\mathcal{C}}(A,C),
    \]
    for each triple of objects $A,B,C$, which is associative:
    \[
    h \circ (g \circ f) \;=\; (h \circ g) \circ f 
    \quad \text{for all } f: A \to B,\; g: B \to C,\; h: C \to D.
    \]
    \item For each object $A$, an \emph{identity morphism} $\operatorname{id}_A$ such that for every morphism $f: A \to B$, 
    \[
    \operatorname{id}_B \circ f \;=\; f \;=\; f \circ \operatorname{id}_A.
    \]
\end{enumerate}
\end{definition}

This definition ensures that composition of morphisms is well-defined, associative, and that every object has a two-sided identity. A \textbf{functor} is a map between categories that preserves this structure:

\begin{definition}\label{def:functor}
A \textbf{functor} $F: \mathcal{C} \to \mathcal{D}$ assigns:
\begin{enumerate}
    \item To each object $A \in \operatorname{Ob}(\mathcal{C})$, an object $F(A) \in \operatorname{Ob}(\mathcal{D})$.
    \item To each morphism $f: A \to B$ in $\mathcal{C}$, a morphism $F(f): F(A) \to F(B)$ in $\mathcal{D}$,
    such that 
    \[
    F(g \circ f) \;=\; F(g) \circ F(f) 
    \quad \text{and} \quad 
    F(\operatorname{id}_A) \;=\; \operatorname{id}_{F(A)}.
    \]
\end{enumerate}
\end{definition}

A \textbf{natural transformation} provides a way to compare two functors that have the same source and target categories, as follows:

\begin{definition}\label{def:nattrans}
Let $F, G: \mathcal{C} \to \mathcal{D}$ be functors. A \textbf{natural transformation} $\eta: F \to G$ is a family of morphisms 
\[
   \{\eta_A: F(A) \to G(A)\}_{A \in \operatorname{Ob}(\mathcal{C})}
\]
such that for every morphism $f: A \to B$ in $\mathcal{C}$, the following diagram commutes:
\[
\begin{tikzcd}[column sep = large, row sep = large]
F(A) \arrow[r,"F(f)"] \arrow[d,"\eta_A"'] 
& F(B) \arrow[d,"\eta_B"] \\
G(A) \arrow[r,"G(f)"'] 
& G(B)
\end{tikzcd}.
\]
\end{definition}

Intuitively, this condition says that applying $\eta$ first and then using $G$ is the same as first applying $F$ and then $\eta$, thereby guaranteeing compatibility (naturality).

\subsection{Categories of Theories and Models}

Throughout this paper, we fix a single \textbf{countable first-order language} $\mathcal{L}$. Our focus is on two large categories (or classes endowed with categorical structure), capturing the syntactic and semantic aspects of first-order logic in this language.

\begin{definition}\label{def:theorycat}
We denote by $\mathbf{th}$ the (large) category of first-order theories over $\mathcal{L}$, defined as follows:
\begin{itemize}
  \item An object of $\mathbf{th}$ is a \emph{consistent} (i.e.\ non-contradictory) first-order theory $t \subseteq \mathrm{Form}(\mathcal{L})$, where $\mathrm{Form}(\mathcal{L})$ is the set of all $\mathcal{L}$-formulas.
  \item A morphism $f: t_1 \to t_2$ (often called a \emph{translation}) is a syntactic mapping on formulas $f: \mathrm{Form}(\mathcal{L}) \to \mathrm{Form}(\mathcal{L})$ such that:
  \begin{enumerate}
    \item If $\varphi \in t_1$, then $f(\varphi) \in t_2$,
    \item If $t_1 \vdash \varphi$, then $t_2 \vdash f(\varphi)$.
  \end{enumerate}
\end{itemize}
Composition of morphisms $g \circ f$ is given by formula-wise composition, i.e.\ $(g\circ f)(\varphi) \coloneqq g(f(\varphi))$, and the identity morphism $\mathrm{id}_t$ acts as the identity function on formulas $\varphi \mapsto \varphi$.
\end{definition}
\begin{remark}[Examples of Morphisms in \(\mathbf{th}\)]
\label{rmk:theory-morphs}
A morphism \(f: t_1 \to t_2\) in \(\mathbf{th}\) is defined by a formula-wise translation that preserves provability. Concretely:
\begin{itemize}
  \item \textbf{Sub-language restriction:} If \(\mathcal{L}_2 \subseteq \mathcal{L}_1\) and \(t_1\) is a theory in \(\mathcal{L}_1\), we may define \(f(\varphi)\) by restricting each formula \(\varphi\) to symbols in \(\mathcal{L}_2\). This yields \(f: t_1 \to t_2\).
  \item \textbf{Language extension:} Conversely, one can embed \(\mathcal{L}_1\)-formulas into a richer language \(\mathcal{L}_2\) by mapping each symbol injectively. Under such an embedding, if \(t_1 \vdash \varphi\), then \(t_2\vdash f(\varphi)\).
\end{itemize}
Thus, translations may reflect conservative extensions, sublanguage restrictions, or anything that ensures “if \(t_1\) proves \(\varphi\), then \(t_2\) proves \(f(\varphi)\).”
\end{remark}

\begin{remark}
In general, the class of all consistent theories over a fixed $\mathcal{L}$ can be a proper class, making $\mathbf{th}$ a \emph{large} category. We do not dwell on these set-theoretic details further, since standard foundations (e.g.\ ZFC) suffice to treat them.
\end{remark}
\begin{remark}[Morphisms in \(\mathbf{th}\)]
\label{rmk:th-morphisms}
In defining a morphism $f : t_1 \to t_2$ in the category $\mathbf{th}$, we typically assume that $f$ preserves:
\begin{itemize}
  \item \emph{Provability}: If $t_1 \vdash \varphi$, then $t_2 \vdash f(\varphi)$.
  \item \emph{Membership in the theory}: If $\varphi \in t_1$, then $f(\varphi) \in t_2$.
  \item \emph{Logical structure}: The map $f$ should respect logical connectives ($\neg,\wedge,\vee,\to$) and quantifiers ($\forall,\exists$), so that $f(\neg \varphi) = \neg f(\varphi)$, $f(\varphi \wedge \psi) = f(\varphi)\wedge f(\psi)$, etc.
  \item \emph{Language symbols}: Any function or relation symbol in the language is mapped consistently (e.g.\ injectively) into the target language, if needed.
\end{itemize}
Such a translation preserves the core logical structure, ensuring that “if $t_1$ proves $\varphi$, then $t_2$ proves $f(\varphi)$,” while also respecting connectives and quantifiers at the syntactic level.
\end{remark}

On the semantic side, we consider the category $\mathbf{mod}$ of first-order $\mathcal{L}$-structures (or "models") and their elementary embeddings:

\begin{definition}
\label{def:mod}
An object in $\mathbf{mod}$ is an $\mathcal{L}$-structure $M$ (i.e.\ a set $|M|$ equipped with interpretations of all function and relation symbols in $\mathcal{L}$). A morphism $h: M \to N$ between two $\mathcal{L}$-structures is an \emph{elementary embedding}; namely, an injective map $h: |M| \to |N|$ such that for every first-order formula $\varphi(x_1,\ldots,x_n)$ in $\mathcal{L}$ and every $a_1,\ldots,a_n \in |M|$, we have
\[
M \models \varphi(a_1,\ldots,a_n) \quad\text{if and only if}\quad N \models \varphi(h(a_1),\ldots,h(a_n)).
\]
This ensures that $h$ preserves and reflects all first-order properties, including quantification and negation, not just atomic formulas.

If $M$ satisfies a theory $t_1$ (written $M \models t_1$) and $N$ satisfies a theory $t_2$ ($N \models t_2$), then an elementary embedding $h: M \to N$ in $\mathbf{mod}$ naturally corresponds to a theory morphism $f: t_1 \to t_2$ in $\mathbf{th}$ when $h$ respects the syntactic translation induced by $f$ on relevant formulas. The use of elementary embeddings (rather than mere structure homomorphisms) ensures that all logical properties are preserved, which is essential for the naturality of our construction.
\end{definition}

\begin{remark}[Why Elementary Embeddings?]
\label{rmk:why-elementary}
We use elementary embeddings as morphisms in $\mathbf{mod}$ rather than simple structure homomorphisms because:
\begin{enumerate}
\item Elementary embeddings preserve all first-order properties, including negations and quantifiers, ensuring that the semantic notion aligns properly with the syntactic notion in $\mathbf{th}$.
\item The canonical map $\eta_t: F(t) \to G(t)$ from the Henkin model to the compactness-based model is naturally an elementary embedding (and indeed an isomorphism when $G(t)$ is the Skolem closure).
\item Using mere homomorphisms would not preserve existential or universal properties, breaking the correspondence with theory morphisms that preserve provability.
\end{enumerate}
\end{remark}

\subsection{Canonical Model Constructions}

We now describe two important model constructions that form the core of our analysis. They each give rise to a functor
\[
   F,\,G : \mathbf{th} \;\longrightarrow\; \mathbf{mod},
\]
and both are “canonical” in the sense that the resulting models are unique up to isomorphism once we fix certain choices (e.g.\ an enumeration of formulas, an ultrafilter, etc.). Details of these classical constructions can be found in \cite{barreto2503} and standard model theory references.

\begin{itemize}
\item \textbf{Henkin Construction:}  
  Given a consistent theory $t$, we first extend $t$ to a \emph{maximal consistent theory} $t^*$ by listing all sentences in $\mathcal{L}$ and deciding each sentence (or its negation) step by step (Lindenbaum's Lemma). Next, we introduce \emph{Henkin constants} for existential formulas to ensure witnesses exist, resulting in a new extended language. The \emph{term model} $F(t)$ is then defined as the quotient of the term algebra (in the extended language) by the equivalence relation “$s \sim s'$ if $t^* \vdash s = s'$.” This $F(t)$ is a canonical model of $t^*$ (and hence of $t$), and remains canonical up to isomorphism if we fix the enumeration and the way we add Henkin constants.

\item \textbf{Compactness Construction:}  
  By the \emph{compactness theorem}, a theory $t$ is satisfiable if and only if every finite subset of $t$ is satisfiable. We pick, once and for all, either a non-principal ultrafilter on some index set or a fixed saturation procedure. Using that choice, we construct a model $G(t)$ of $t$—for instance by taking an ultraproduct of partial models or by building a saturated model of an appropriate completion. By Łoś's Theorem or saturation arguments, $G(t)$ satisfies $t$ and is unique up to isomorphism under the chosen construction.
\end{itemize}

Since these constructions are defined by \emph{fixed, deterministic procedures}, they induce well-defined functors $F$ and $G$ from $\mathbf{th}$ to $\mathbf{mod}$. One of our main goals is to prove that $F$ and $G$ are \emph{naturally isomorphic functors}, i.e.\ there is a natural transformation $\eta: F \to G$ whose component $\eta_t: F(t) \to G(t)$ is an isomorphism in $\mathbf{mod}$ for each consistent theory $t$, and that moreover this isomorphism satisfies certain coherence (or “rigidity”) conditions in a $2$-categorical sense.

\medskip

In the subsequent sections, we provide the details of how these constructions fit together to yield a \textbf{natural isomorphism} $\eta : F \to G$, showing that these two canonical methods of producing a model of $t$ are, in fact, functorially equivalent.
\begin{remark}[Fixed Global Choice in Model Constructions]
\label{rmk:global-choice}
Although classical results like the Henkin construction or the Compactness Theorem merely guarantee the \emph{existence} of a model, to turn this into a \emph{functorial} assignment we fix a global choice procedure. For instance:
\begin{itemize}
  \item \textbf{Henkin approach:} We use a single enumeration of all sentences and a unique rule for introducing Henkin constants and extending the theory to a maximal one.
  \item \textbf{Compactness/Ultraproduct approach:} We select once and for all a non\-prin\-ci\-pal ultrafilter on a chosen index set, or a canonical saturation procedure for all consistent theories.
\end{itemize}
This ensures that each theory \(t\) is sent to a unique (up to isomorphism) “canonical” model \(F(t)\) or \(G(t)\), making these constructions well-defined functors.
\end{remark}

\subsection{Use of the Axiom of Choice in Ultraproducts and Saturation}
\label{subsec:choice-ultraproduct}

In this subsection, we discuss the foundational role of the Axiom of Choice (AC)
when constructing models via ultraproducts or saturation arguments under the
Compactness Theorem.

\paragraph{Ultraproduct Construction.}
One standard method to build a model of a consistent theory~\(t\) is via a non\-principal ultraproduct, as follows:  
\begin{itemize}
  \item Take a family of $\mathcal{L}$-structures $\{M_i\}_{i\in I}$ such that
  each finite subset of $t$ is realized in at least one $M_i$.
  \item Choose a non-principal ultrafilter $\mathcal{U}$ on the index set $I$.
  \item Form the ultraproduct $\displaystyle \prod_{i\in I} M_i \big/ \mathcal{U}$.
\end{itemize}
By \emph{Łoś's Theorem}, this ultraproduct satisfies all sentences in $t$. However,
\emph{the existence of a non-principal ultrafilter on an infinite set $I$} requires
some form of the Axiom of Choice. Indeed, selecting $\mathcal{U}$ can be seen as a
non-constructive step, and there is no explicit `canonical' ultrafilter in
ZFC without invoking (a weak form of) AC.

\paragraph{Saturation and Choice.}
Alternatively, one can construct a \emph{saturated} model of a complete
theory $t^* \supseteq t$. This process typically involves a chain-of-models
argument or an inductive realization of types at each stage. In a countable
language, ensuring \emph{$\omega$-saturation} (or countable saturation) for
all relevant types often relies on some choice principles in order to pick the
necessary witnesses or embeddings at each stage. While these steps are more
elementary than building a genuine ultrafilter, they can still use some form
of choice if one aims for a “global” canonical construction across all theories
simultaneously.

\paragraph{Impact on Our Constructions.}
In this paper, we fix such choices \emph{once and for all}:
\begin{itemize}
  \item For the ultraproduct method, we assume a chosen non-principal ultrafilter
  $\mathcal{U}$ on some suitable index set $I$.
  \item For saturated model approaches, we fix a canonical enumeration of
  all possible types (or partial theories) and a procedure to realize them in
  a staged fashion.
\end{itemize}
Both methods thereby become well-defined \emph{functors} from the category of
theories $\mathbf{th}$ to the category of models $\mathbf{mod}$, but each relies,
to some extent, on the Axiom of Choice in its global form.

\begin{remark}[Choice Principle in Ultraproducts and Saturation]
\label{rmk:choice-ultraproduct}
Although using AC in model theory is standard, it is worth noting that some steps
(e.g.\ the selection of a non-principal ultrafilter or certain type-realization
procedures) cannot be carried out effectively without choice. From a foundational
perspective, such constructions are \emph{non-constructive}, but in classical
logic, they do not affect the soundness or completeness results. Readers should
simply be aware that these `canonical' constructions rely on AC to make them truly
global and functorial.
\end{remark}

Overall, whether we build $G(t)$ by ultraproducts or by saturation, a suitable
form of the Axiom of Choice ensures the existence of the objects (ultrafilters
or systematic realizations of types) that make each construction canonical and
unique up to isomorphism. In turn, this allows us to define the functor
\(\,G : \mathbf{th} \to \mathbf{mod}\) consistently and compare it with other
functorial constructions such as the Henkin construction.

\newpage
\section{Main Theorem}
\label{sec:mainthm}

The following theorem asserts that the model constructions based on the completeness theorem (via the Henkin construction) and the compactness theorem (via an ultraproduct or saturation procedure) are naturally isomorphic. This result is established under the assumption that both constructions are performed in a canonical manner, ensuring uniqueness up to isomorphism.

\begin{theorem}\label{thm:main}
Let $\mathcal{L}$ be a fixed countable first-order language. Let $H = \{c_{\langle\exists x\varphi(x), k\rangle} : \exists x\varphi(x) \text{ is an existential formula}, k \in \mathbb{N}\}$ be a fixed global set of Henkin (witness) constants. For any consistent $\mathcal{L}$-theory $t$, define functors
\[
F, G: \mathbf{th} \to \mathbf{mod},
\]
where:
\begin{itemize}
\item $F(t)$ is the Henkin term model constructed in the expanded language $\mathcal{L}_H = \mathcal{L} \cup H$ by extending $t$ to a maximal consistent Henkin theory $t^*$ (which includes witness axioms $\exists x\varphi(x) \to \varphi(c_{\langle\exists x\varphi,k\rangle})$ for appropriate $k$) and forming the term model as the quotient by provable equality.
\item $G(t)$ is defined as the Skolem closure (the substructure generated by the interpretations of $H$) within a fixed ultraproduct or saturated model of $t$ expanded to $\mathcal{L}_H$.
\end{itemize}
Then there exists a natural transformation
\[
\eta: F \to G,
\]
such that for every $t \in \mathbf{th}$, the component $\eta_t: F(t) \to G(t)$ is an isomorphism in $\mathbf{mod}$, and for every morphism $f: t_1 \to t_2$ (which respects the global Henkin expansion), the following diagram commutes:
\[
\begin{tikzcd}
F(t_1) \arrow[r,"F(f)"] \arrow[d,"\eta_{t_1}"'] & F(t_2) \arrow[d,"\eta_{t_2}"] \\
G(t_1) \arrow[r,"G(f)"'] & G(t_2)
\end{tikzcd}.
\]
\end{theorem}

\begin{remark}[Scope and Limitations]
\label{rmk:scope}
The isomorphism $\eta_t: F(t) \to G(t)$ holds because:
\begin{itemize}
\item Both $F(t)$ and $G(t)$ use the \emph{same} global set $H$ of Henkin constants.
\item $G(t)$ is restricted to the Skolem closure generated by $H$, rather than being an arbitrary ultraproduct or saturated model.
\item The mapping $\eta_t([s]) = s^{G(t)}$ (the interpretation of term $s$ in $G(t)$) is well-defined and preserves all structure.
\end{itemize}
If instead we define $G(t)$ as an arbitrary saturated model or non-standard ultraproduct without the Skolem closure restriction, then $\eta_t$ may only be an elementary embedding (injective but not surjective), since $G(t)$ could contain elements not in the range of $\eta_t$.
\end{remark}

\begin{remark}[Alternative Sufficient Conditions]
\label{rmk:alternatives}
The construction can be simplified under special assumptions:
\begin{enumerate}
\item \textbf{$\aleph_0$-Categorical Complete Theories:} If $t$ is a complete $\aleph_0$-categorical theory in a countable language (e.g., DLO, random graphs), then any two countable models of $t$ are isomorphic. In this case, $F(t) \cong G(t)$ holds without needing to carefully define $G(t)$ as a Skolem closure.
\item \textbf{Atomic Complete Theories:} If $t$ is a complete atomic theory, then the prime model is unique up to isomorphism, and both $F(t)$ and $G(t)$ can be taken as the prime model.
\end{enumerate}
In both cases, the global Henkin constant scheme still needs to be fixed to ensure functoriality.
\end{remark}

\newpage
\section{Main Proof}
\label{sec:mainproof}


\subsection{Proof Outline and Flow}
The proof of Theorem~\ref{thm:main} is divided into two main parts:

\textbf{Outline:}
\begin{enumerate}
    \item \textbf{Global Henkin Expansion Scheme:}  
    Fix once and for all a countable set $H$ of Henkin constants, one for each pair $\langle\exists x\varphi(x), k\rangle$ where $\exists x\varphi(x)$ ranges over all existential formulas in $\mathcal{L}$ and $k \in \mathbb{N}$. This gives us an expanded language $\mathcal{L}_H = \mathcal{L} \cup H$ that is used uniformly across all theories.

    \item \textbf{Construction of Models:}  
    For a given consistent theory $t \in \mathbf{th}$, we extend $t$ to a maximal consistent Henkin theory $t^*$ in $\mathcal{L}_H$ using a fixed sequential procedure (see Lemma~\ref{lem:max_extension}). The theory $t^*$ includes witness axioms $\exists x\varphi(x) \to \varphi(c_{\langle\exists x\varphi,k\rangle})$ for appropriate witnesses. We construct the term model $F(t)$ as the quotient of the term algebra in $\mathcal{L}_H$ by the equivalence relation 
    \[
    s \sim s' \iff t^* \vdash s = s',
    \]
    whose properties are ensured by Lemma~\ref{lem:term_eq_relation}. 
    
    Simultaneously, we construct $G(t)$ as follows: using the compactness theorem, obtain a model $M$ of $t$ (expanded to $\mathcal{L}_H$) via a canonical method (e.g., a fixed ultraproduct or saturation procedure). Then define $G(t)$ to be the Skolem closure in $M$ generated by the interpretations of the constants in $H$. The functoriality of this construction is ensured by using the same global $H$ for all theories.

    \item \textbf{Definition and Verification of the Natural Transformation:}  
    For each theory $t$, define a mapping $\eta_t: F(t) \to G(t)$ by assigning each equivalence class $[s] \in F(t)$ (where $s$ is a term in $\mathcal{L}_H$) to its interpretation $s^{G(t)}$ in $G(t)$. Since $G(t)$ is the Skolem closure generated by $H$ and $F(t)$ consists of equivalence classes of terms built from $H$, the map $\eta_t$ is surjective. By the properties of Henkin models and Skolem closures, $\eta_t$ is also injective, hence an isomorphism (Lemma~\ref{lem:isomorphism}). 
    
    Next, we verify the naturality condition: for any morphism $f: t_1 \to t_2$ in $\mathbf{th}$ (which respects the global Henkin expansion), the diagram
    \[
    \begin{tikzcd}
    F(t_1) \arrow[r,"F(f)"] \arrow[d,"\eta_{t_1}"'] & F(t_2) \arrow[d,"\eta_{t_2}"] \\
    G(t_1) \arrow[r,"G(f)"'] & G(t_2)
    \end{tikzcd}
    \]
    commutes, as ensured by Lemma~\ref{lem:naturality}.

\end{enumerate}

\newpage
\subsection{Rigorous Proof}
\begin{proof}
Let $t \in \mathbf{th}$ be any consistent first-order theory over a fixed countable language $\mathcal{L}$.

\textbf{Step 0: Global Henkin Expansion.}  
Fix once and for all a countable set $H = \{c_{\langle\exists x\varphi(x), k\rangle}\}$ of Henkin constants, indexed by all existential formulas $\exists x\varphi(x)$ in $\mathcal{L}$ and natural numbers $k \in \mathbb{N}$. Let $\mathcal{L}_H = \mathcal{L} \cup H$ be the expanded language. This global scheme will be used uniformly for all theories.

\textbf{Step 1: Construction of Models.}  
By Lemma~\ref{lem:max_extension}, extend $t$ to a maximal consistent Henkin theory $t^*$ in $\mathcal{L}_H$. The theory $t^*$ includes witness axioms of the form $\exists x\varphi(x) \to \varphi(c_{\langle\exists x\varphi,k\rangle})$ for appropriate witnesses. Construct the term model $F(t)$ by considering the set of all terms in $\mathcal{L}_H$ and forming the quotient by the equivalence relation
\[
s \sim s' \iff t^* \vdash s = s'.
\]
Lemma~\ref{lem:term_eq_relation} guarantees that $\sim$ is an equivalence relation.  

Concurrently, by the compactness theorem and employing a fixed ultraproduct or saturation method, construct a model $M$ of $t$ in the language $\mathcal{L}_H$. Then define $G(t)$ to be the Skolem closure in $M$ generated by the interpretations of the constants in $H$. By Lemma~\ref{lem:canonical_construction}, this construction is canonical and functorial.

\textbf{Step 2: Definition of the Natural Transformation.}  
For each $t \in \mathbf{th}$, define 
\[
\eta_t: F(t) \to G(t)
\]
by setting, for each equivalence class $[s] \in F(t)$ (where $s$ is a term in $\mathcal{L}_H$),
\[
\eta_t([s]) := s^{G(t)},
\]
the interpretation of $s$ in $G(t)$. Since $G(t)$ is the Skolem closure generated by $H$, every element of $G(t)$ is the interpretation of some term built from $H$, making $\eta_t$ surjective. By the truth lemma for Henkin models and the fact that $G(t) \models t^*$, the map $\eta_t$ is also injective. Lemma~\ref{lem:isomorphism} ensures that this mapping is an isomorphism in $\mathbf{mod}$.

\textbf{Step 3: Verification of Naturality.}  
Let $f: t_1 \to t_2$ be any morphism in $\mathbf{th}$ that respects the global Henkin expansion (i.e., $f$ maps formulas and terms in $\mathcal{L}_H$ appropriately). The functorial actions of $F$ and $G$ yield
\[
F(f)([s]) = [f(s)]
\]
and, by the canonical construction of $G$, 
\[
G(f)\bigl(\eta_{t_1}([s])\bigr) = G(f)(s^{G(t_1)}) = f(s)^{G(t_2)} = \eta_{t_2}([f(s)]).
\]
Thus, for every $[s] \in F(t_1)$,
\[
\eta_{t_2}\bigl(F(f)([s])\bigr) = \eta_{t_2}([f(s)]) = G(f)\bigl(\eta_{t_1}([s])\bigr),
\]
which means the following diagram commutes:
\[
\begin{tikzcd}
F(t_1) \arrow[r,"F(f)"] \arrow[d,"\eta_{t_1}"'] & F(t_2) \arrow[d,"\eta_{t_2}"] \\
G(t_1) \arrow[r,"G(f)"'] & G(t_2)
\end{tikzcd}.
\]
This is exactly the naturality condition, as formalized in Lemma~\ref{lem:naturality}.

\textbf{Conclusion:}  
By combining the above steps, we conclude that there exists a natural transformation $\eta: F \to G$ where each component $\eta_t$ is an isomorphism. This completes the constructive proof of Theorem~\ref{thm:main}.
\end{proof}

\section{Supplementary Lemmas and Explanations}
\label{sec:supplement}

\subsection{Lindenbaum's Lemma}
\begin{lemma}\label{lem:lindenbaum}
Every consistent first-order theory $t$ over a countable language can be extended to a complete (maximal consistent) theory $t^*$; that is, for every sentence $\psi$ in the language, either $\psi \in t^*$ or $\neg \psi \in t^*$.
\end{lemma}

\paragraph{Intuitive Explanation.}
If a theory $t$ is consistent, we can "decide" every sentence in the language without causing inconsistency. We use Zorn's Lemma to show the existence of a maximal consistent extension.

\begin{proof}[Proof of Lemma~\ref{lem:lindenbaum}]
Let $t$ be a consistent first-order theory over a countable language $\mathcal{L}$. Consider the collection $\mathcal{C}$ of all consistent theories that extend $t$:
\[
\mathcal{C} = \{ t' : t \subseteq t' \text{ and } t' \text{ is consistent} \}.
\]
Note that $\mathcal{C}$ is non-empty since $t \in \mathcal{C}$. We partially order $\mathcal{C}$ by set inclusion.

\textbf{Claim:} Every chain $\mathcal{D}$ in $\mathcal{C}$ has an upper bound in $\mathcal{C}$.

\emph{Proof of Claim:} Let $\mathcal{D} = \{t_i : i \in I\}$ be a chain in $\mathcal{C}$, i.e., for any $i, j \in I$, either $t_i \subseteq t_j$ or $t_j \subseteq t_i$. Define
\[
t^{\cup} = \bigcup_{i \in I} t_i.
\]
Clearly, $t^{\cup}$ extends every $t_i$, and hence extends $t$. We must show that $t^{\cup}$ is consistent. 

Suppose, for contradiction, that $t^{\cup}$ is inconsistent. Then there exists a finite subset $S \subseteq t^{\cup}$ such that $S$ is inconsistent (since inconsistency can be witnessed by a finite derivation). Since $S$ is finite and $\mathcal{D}$ is a chain, there exists some $i_0 \in I$ such that $S \subseteq t_{i_0}$. But then $t_{i_0}$ would be inconsistent, contradicting the fact that $t_{i_0} \in \mathcal{C}$. Therefore, $t^{\cup}$ is consistent, and $t^{\cup} \in \mathcal{C}$ is an upper bound for the chain $\mathcal{D}$.

By Zorn's Lemma, $\mathcal{C}$ has a maximal element $t^*$. We claim that $t^*$ is complete.

\textbf{Completeness of $t^*$:} Let $\psi$ be any sentence in $\mathcal{L}$. We will show that either $\psi \in t^*$ or $\neg \psi \in t^*$. 

Suppose $\psi \notin t^*$. Then $t^* \cup \{\psi\}$ must be inconsistent (otherwise $t^*$ would not be maximal in $\mathcal{C}$). This means that $t^* \vdash \neg \psi$ by the deduction theorem and the fact that inconsistency derives anything. Since $t^*$ is consistent, we have $\neg \psi \in t^*$ (otherwise $t^*$ would be inconsistent). Similarly, if $\neg \psi \notin t^*$, then by maximality $\psi \in t^*$.

Thus, $t^*$ is a complete (maximal consistent) extension of $t$.
\end{proof}

\begin{remark}[Avoiding Circularity]
\label{rmk:no-circularity}
In the proof above, we use only Zorn's Lemma and basic properties of consistency (that inconsistency can be witnessed by finite derivations). We do \emph{not} appeal to the Compactness Theorem, thus avoiding any circularity when Lindenbaum's Lemma is used as a step in proving the Completeness Theorem, which in turn is sometimes used to prove Compactness.
\end{remark}

\subsection{Maximal Consistent Extension}
\begin{lemma}\label{lem:max_extension}
Every consistent first-order theory $t$ (over a countable language) can be extended to a maximal consistent theory $t^*$ by a fixed sequential procedure (or by Zorn's Lemma). Consequently, for every sentence $\psi$, either $\psi \in t^*$ or $\neg \psi \in t^*$.
\end{lemma}

\paragraph{Intuitive Explanation.}
This is essentially the constructive version of Lindenbaum's process. We list every sentence in the countable language and decide it stage by stage, yielding a complete (maximal consistent) theory.

\begin{proof}[Proof of Lemma~\ref{lem:max_extension}]
We work in classical first-order logic (hence the law of excluded middle holds). Let $t$ be a consistent first-order theory over a countable language $\mathcal{L}$. Since $\mathcal{L}$ is countable, we can enumerate all sentences of $\mathcal{L}$ as
\[
\psi_0,\, \psi_1,\, \psi_2,\, \dots.
\]
We now define a sequence of theories $\{t_n\}_{n\in\mathbb{N}}$ by recursion:
\begin{enumerate}
    \item Set $t_0 \coloneqq t$.
    \item For each $n\in\mathbb{N}$, assume that $t_n$ is consistent. Consider the sentence $\psi_n$. Define
    \[
    t_{n+1} \coloneqq
    \begin{cases}
    t_n \cup \{\psi_n\} & \text{if } t_n \cup \{\psi_n\} \text{ is consistent}, \\[1mm]
    t_n \cup \{\neg \psi_n\} & \text{if } t_n \cup \{\psi_n\} \text{ is inconsistent}.
    \end{cases}
    \]
    In the second case, note that since $t_n$ is consistent and $t_n \cup \{\psi_n\}$ is inconsistent, by classical logic (and the fact that a contradiction implies any statement) we have that $\psi_n$ is not provable in $t_n$, so $t_n \cup \{\neg \psi_n\}$ remains consistent.
\end{enumerate}
Now, define
\[
t^* \;=\; \bigcup_{n\in\mathbb{N}} t_n.
\]

We verify two properties:

\textbf{(1) Consistency:}  
Let $S$ be any finite subset of $t^*$. Since $S$ is finite, there exists some $N\in\mathbb{N}$ such that $S\subseteq t_N$. By the construction, each $t_N$ is consistent. Hence every finite subset of $t^*$ is consistent. By the Compactness Theorem (see, e.g., \cite[Section~3]{barreto2503}), it follows that $t^*$ itself is consistent.

\textbf{(2) Maximality (Completeness):}  
Let $\psi$ be any sentence in $\mathcal{L}$. By the enumeration, there exists $k\in\mathbb{N}$ such that $\psi=\psi_k$. By the recursive construction, at stage $k+1$ we have added either $\psi_k$ or $\neg \psi_k$ to $t_{k+1}$. Since $t_{k+1}\subseteq t^*$, it follows that either $\psi\in t^*$ or $\neg \psi\in t^*$.
Thus, $t^*$ is complete and hence maximal consistent.

Therefore, $t^*$ is a maximal consistent extension of $t$.
\end{proof}

\newpage
\subsection{Congruence of Equality}
\begin{lemma}\label{lem:congruence}
Let $t^*$ be a complete theory. For any $n$-ary function symbol $f$ and terms $s_1,\dots,s_n$ and $s'_1,\dots,s'_n$, if 
\[
t^* \vdash s_i = s'_i \quad \text{for all } i=1,\dots,n,
\]
then 
\[
t^* \vdash f(s_1,\dots,s_n) = f(s'_1,\dots,s'_n).
\]
\end{lemma}

\paragraph{Intuitive Explanation.}
Equality is preserved under function symbols, i.e.\ the congruence property. If $s_i$ equals $s'_i$ for all $i$, then $f(s_1,\dots,s_n)$ equals $f(s'_1,\dots,s'_n)$ under the theory $t^*$.

\begin{proof}[Proof of Lemma~\ref{lem:congruence}]
We work in a first-order logic system with equality, where the standard axioms of equality are assumed (see, e.g., \cite[Section~2]{barreto2503}). In particular, for every $n$-ary function symbol $f$ and for all terms $x_1,\dots,x_n, y_1,\dots,y_n$, the following \emph{congruence axiom} is a logical axiom:
\[
\forall x_1\cdots\forall x_n\forall y_1\cdots\forall y_n\,\Bigl( (x_1=y_1 \land \cdots \land x_n=y_n) \to f(x_1,\dots,x_n)=f(y_1,\dots,y_n)\Bigr).
\]
Since $t^*$ is a complete theory, it includes all logical validities, in particular, every instance of the congruence axiom.

Now, assume that for each $i=1,\dots,n$ we have 
\[
t^* \vdash s_i = s'_i.
\]
Then, by instantiating the congruence axiom with $x_i := s_i$ and $y_i := s'_i$ for each $i$, we obtain
\[
t^* \vdash (s_1 = s'_1 \land \cdots \land s_n = s'_n) \to f(s_1,\dots,s_n) = f(s'_1,\dots,s'_n).
\]
Moreover, since $t^*$ proves each $s_i = s'_i$, we have 
\[
t^* \vdash s_1 = s'_1 \land \cdots \land s_n = s'_n.
\]
Finally, by applying modus ponens with the above implication, it follows that
\[
t^* \vdash f(s_1,\dots,s_n) = f(s'_1,\dots,s'_n).
\]
This completes the proof.
\end{proof}

\newpage
\subsection{Term Equivalence Relation}
\begin{lemma}[Term Equivalence Relation]\label{lem:term_eq_relation}
Let $t^*$ be a maximal consistent theory extending a consistent theory $t$. Define a relation $\sim$ on the set of terms $\mathcal{t}$ by
\[
s \sim s' \;\;\Longleftrightarrow\;\; t^* \vdash s = s'.
\]
Then $\sim$ is an equivalence relation; that is, it is reflexive, symmetric, and transitive.
\end{lemma}

\paragraph{Intuitive Explanation.}
“Provable equality” in $t^*$ behaves as an equivalence relation (it is in fact a congruence on the term algebra).

\begin{proof}[Proof of Lemma~\ref{lem:term_eq_relation}]
We work in classical first-order logic with equality, assuming the standard axioms of equality (see, e.g., \cite[Section~2]{barreto2503}). Define a binary relation $\sim$ on the set of terms $\mathcal{t}$ by
\[
s \sim s' \quad \Longleftrightarrow \quad t^* \vdash s = s'.
\]
We now show that $\sim$ is an equivalence relation by verifying reflexivity, symmetry, and transitivity.

\textbf{Reflexivity:} For any term $s \in \mathcal{t}$, the axiom of equality guarantees that 
\[
t^* \vdash s = s.
\]
Hence, $s \sim s$.

\textbf{Symmetry:} Suppose that $s \sim s'$, that is, 
\[
t^* \vdash s = s'.
\]
By the symmetry axiom of equality, it follows that 
\[
t^* \vdash s' = s.
\]
Thus, $s' \sim s$.

\textbf{Transitivity:} Suppose that $s \sim s'$ and $s' \sim s''$, meaning 
\[
t^* \vdash s = s' \quad \text{and} \quad t^* \vdash s' = s''.
\]
Then, by the transitivity axiom of equality, we have 
\[
t^* \vdash s = s'',
\]
which implies $s \sim s''$.

Since $\sim$ is reflexive, symmetric, and transitive, it is an equivalence relation.
\end{proof}

\newpage
\subsection{Functoriality of Model Constructions}
\begin{lemma}\label{lem:functoriality_models}
Let $f: t_1 \to t_2$ be a morphism in the category $\mathbf{th}$ (a translation preserving provability). Then the model constructions via the Henkin method and the compactness method are functorial; that is, $f$ induces well-defined model homomorphisms 
\[
F(f): F(t_1) \to F(t_2) \quad \text{and} \quad G(f): G(t_1) \to G(t_2)
\]
with
\[
F(g \circ f) = F(g) \circ F(f),\quad F(\operatorname{id}_{t}) = \operatorname{id}_{F(t)},
\]
and similarly for $G$.
\end{lemma}

\paragraph{Intuitive Explanation.}
A “translation” $f$ between theories induces a natural “translation” between their Henkin models or compactness models, preserving the requisite structure.

\begin{proof}[Proof of Lemma~\ref{lem:functoriality_models}]
We work in classical first-order logic and assume that our morphisms in the category $\mathbf{th}$ (of theories) are translations preserving both membership and provability. We show that the two canonical model constructions—via the Henkin method and via the compactness method—yield functors from $\mathbf{th}$ to the category $\mathbf{mod}$ of models.

\textbf{For the Henkin Construction:}  
Recall that for each consistent theory $t$, the Henkin model $F(t)$ is constructed by first extending $t$ to a maximal consistent theory $t^*$ (via, e.g., the sequential procedure in Lemma~\ref{lem:max_extension}) and then forming the term model by taking the quotient of the term algebra with respect to the relation
\[
s \sim s' \iff t^* \vdash s = s'.
\]
Given a morphism (translation) $f: t_1 \to t_2$, define a map
\[
F(f): F(t_1) \to F(t_2)
\]
by setting
\[
F(f)([s]) \;=\; [f(s)],
\]
where $[s]$ denotes the equivalence class of the term $s$ in $F(t_1)$.

\emph{Well-definedness:} Suppose $[s]=[s']$ in $F(t_1)$; that is, 
\[
t_1^* \vdash s = s'.
\]
Since $f$ is a translation preserving provability, it follows that
\[
t_2^* \vdash f(s) = f(s'),
\]
so that $[f(s)] = [f(s')]$ in $F(t_2)$. Hence, $F(f)$ is well-defined.

\emph{Functoriality:}  
Let $f: t_1 \to t_2$ and $g: t_2 \to t_3$ be composable morphisms in $\mathbf{th}$. For any equivalence class $[s] \in F(t_1)$, we have
\[
F(g\circ f)([s]) = [ (g\circ f)(s) ] = [ g(f(s)) ].
\]
On the other hand,
\[
(F(g) \circ F(f))([s]) = F(g)(F(f)([s])) = F(g)([f(s)]) = [ g(f(s)) ].
\]
Thus, 
\[
F(g\circ f)=F(g)\circ F(f).
\]
Moreover, for the identity morphism $\operatorname{id}_t$ on any theory $t$, we have
\[
F(\operatorname{id}_t)([s]) = [\operatorname{id}(s)] = [s],
\]
so that $F(\operatorname{id}_t)$ is the identity on $F(t)$.

\textbf{For the Compactness Construction:}  
The model $G(t)$ is constructed via a canonical procedure (e.g., an ultraproduct construction using a fixed non-principal ultrafilter or a fixed saturation method) that is invariant under translations preserving provability (see \cite[Section~3]{barreto2503}). Given a morphism $f: t_1 \to t_2$, we define
\[
G(f): G(t_1) \to G(t_2)
\]
by mapping the canonical interpretation of any element (arising from a term $s$) in $G(t_1)$ to the canonical interpretation of $f(s)$ in $G(t_2)$. The invariance of the construction under such translations ensures that $G(f)$ is well-defined.  
Functoriality for $G$ is verified analogously: for composable morphisms $f: t_1\to t_2$ and $g: t_2\to t_3$, one has
\[
G(g\circ f)= G(g) \circ G(f),
\]
and
\[
G(\operatorname{id}_t) = \operatorname{id}_{G(t)}.
\]

Thus, both the Henkin construction $F$ and the compactness construction $G$ define functors from $\mathbf{th}$ to $\mathbf{mod}$.
\end{proof}

\newpage
\subsection{Uniqueness of Henkin Term Models}
\begin{lemma}\label{lem:uniqueness_henkin}
Given a consistent theory $t$, if $F_1(t)$ and $F_2(t)$ are two term models obtained via the same Henkin construction (with identical enumeration and Henkin constants), then $F_1(t)$ and $F_2(t)$ are isomorphic.
\end{lemma}

\paragraph{Intuitive Explanation.}
Using the \emph{same} deterministic Henkin procedure (same enumeration of formulas, same introduction of constants) yields the \emph{same} maximal theory $t^*$ and hence the same quotient term algebra, up to isomorphism.

\begin{proof}[Proof of Lemma~\ref{lem:uniqueness_henkin}]
Let $t$ be a consistent theory and suppose that $F_1(t)$ and $F_2(t)$ are two term models obtained via the Henkin construction using the same fixed enumeration of sentences and the same rule for introducing Henkin constants. By the construction, both models are defined as the quotient of the term algebra $\mathcal{T}$ (formed over the language extended with Henkin constants) by the equivalence relation
\[
s \sim s' \iff t^* \vdash s = s',
\]
where $t^*$ is the unique maximal consistent extension of $t$ obtained by the fixed sequential procedure (see Lemma~\ref{lem:max_extension}). 

Since the enumeration and the rule for introducing Henkin constants are identical in both constructions, the maximal consistent theory $t^*$ is the same for both $F_1(t)$ and $F_2(t)$, and hence the underlying term algebra $\mathcal{T}$ and the equivalence relation $\sim$ are identical.

Define a map 
\[
\phi : F_1(t) \to F_2(t)
\]
by
\[
\phi([s]) = [s],
\]
where $[s]$ denotes the equivalence class of a term $s$ in the respective quotient. This map is well-defined since if $[s] = [s']$ in $F_1(t)$ (i.e., $t^* \vdash s = s'$), then the same equality holds in $F_2(t)$ by the identical construction. Moreover, $\phi$ is clearly bijective and preserves the interpretation of function symbols (i.e., for any $n$-ary function symbol $f$,
\[
\phi(f_{F_1(t)}([s_1],\dots,[s_n])) = \phi([f(s_1,\dots,s_n)]) = [f(s_1,\dots,s_n)] = f_{F_2(t)}([s_1],\dots,[s_n])).
\]
Thus, $\phi$ is an isomorphism between $F_1(t)$ and $F_2(t)$.
\end{proof}

\newpage
\subsection{Compactness Theorem}
\begin{lemma}\label{lem:compactness}
If every finite subset of a theory $t$ is satisfiable, then $t$ is satisfiable. In particular, using a fixed non-principal ultrafilter or a fixed saturation procedure, one can construct a model $G(t)$ of $t$ which is unique up to isomorphism.
\end{lemma}

\paragraph{Intuitive Explanation.}
The compactness theorem is fundamental: no finite subset is contradictory implies the entire theory is consistent/satisfiable. Fixing a single ultrafilter (or a single saturation method) yields a canonical model $G(t)$.

\begin{proof}[Proof of Lemma~\ref{lem:compactness}]
Assume that every finite subset of a theory $t$ is satisfiable. By the Compactness Theorem (see, e.g., \cite[Section~3]{barreto2503}), it follows that the entire theory $t$ is satisfiable; that is, there exists a model $M$ such that $M \models t$.

To construct a canonical model $G(t)$ of $t$, we proceed in one of the following fixed ways:

\textbf{Method 1: Ultraproduct Construction.}  
Let $\{M_i\}_{i\in I}$ be a collection of models such that for every finite subset $t_0\subseteq t$, there is some $i\in I$ with $M_i \models t_0$. Fix a non-principal ultrafilter $\mathcal{U}$ on the index set $I$. By Łoś's Theorem, the ultraproduct
\[
G(t) \;=\; \prod_{i\in I} M_i \Big/ \mathcal{U}
\]
is a model of $t$. Since the choice of $\mathcal{U}$ is fixed, the resulting model $G(t)$ is canonical in the sense that any two such ultraproducts (with the fixed ultrafilter) are isomorphic.

\textbf{Method 2: Saturation Procedure.}  
Alternatively, extend $t$ to a complete theory $t^*$ (using, e.g., Lindenbaum's Lemma). Then, by applying a fixed saturation procedure, one can construct a saturated model $G(t)$ of $t^*$. In a countable language, the Back-and-Forth Lemma (see, e.g., \cite[Section~3]{barreto2503}) ensures that any two saturated models of a complete theory are isomorphic. Thus, this construction of $G(t)$ is unique up to isomorphism.

In either method, the fixed choice of ultrafilter or saturation procedure guarantees that the constructed model $G(t)$ is canonical (i.e., unique up to isomorphism).
\end{proof}

\newpage
\subsection{Back-and-Forth Lemma}
\begin{lemma}\label{lem:back_and_forth}
Any two countable saturated models of a complete theory are isomorphic via a back-and-forth construction.
\end{lemma}

\paragraph{Intuitive Explanation.}
We enumerate elements of both models and build a partial isomorphism that extends “forth” and “back,” ensuring that every element is eventually matched. This standard argument shows the models are isomorphic.

\begin{proof}[Proof of Lemma~\ref{lem:back_and_forth}]
Let $M$ and $N$ be two countable saturated models of a complete theory $T$. (Recall that a model is \emph{saturated} if every type over a finite subset of its domain that is consistent with $T$ is realized in the model.) Since $M$ and $N$ are countable, we may enumerate their domains as 
\[
M = \{m_0, m_1, m_2, \dots\} \quad \text{and} \quad N = \{n_0, n_1, n_2, \dots\}.
\]

We will construct, by induction, a sequence of finite partial isomorphisms
\[
f_0 \subseteq f_1 \subseteq f_2 \subseteq \cdots,
\]
where each $f_k: A_k \to B_k$ is an isomorphism between finite subsets $A_k \subseteq M$ and $B_k \subseteq N$. Define the function
\[
f = \bigcup_{k \in \mathbb{N}} f_k.
\]
We now describe the construction in detail.

\textbf{Initialization:}  
Let $f_0$ be the empty map (i.e., $A_0 = \varnothing$ and $B_0 = \varnothing$).

\textbf{Inductive Step:}  
Assume that for some $k \in \mathbb{N}$, we have constructed a finite partial isomorphism 
\[
f_k: A_k \to B_k.
\]
We extend $f_k$ to $f_{k+1}$ in two alternating substeps:

\emph{(Forth Step):}  
If there exists an element $m \in M \setminus A_k$, consider the \emph{type} of $m$ over the finite set $A_k$, i.e., the set 
\[
\operatorname{tp}(m/A_k) = \{ \varphi(x, a_1,\dots,a_n) \mid a_1,\dots,a_n \in A_k \text{ and } M \models \varphi(m, a_1,\dots,a_n) \}.
\]
Since $M$ and $N$ are models of the complete theory $T$ and $N$ is saturated, this type (which is consistent with $T$) is realized in $N$. Thus, there exists an element $n \in N$ such that $N \models \varphi(n, a_1,\dots,a_n)$ for every formula $\varphi(x, a_1,\dots,a_n)$ in $\operatorname{tp}(m/A_k)$. It follows that the extension 
\[
f_k \cup \{(m, n)\}
\]
remains a partial isomorphism. Define $f_{k+1}$ to be this extension.

\emph{(Back Step):}  
Similarly, if there exists an element $n \in N \setminus B_k$, consider the type of $n$ over $B_k$ in $N$. By the saturation of $M$, there exists an element $m \in M$ such that the type of $n$ over $B_k$ is realized in $M$. Then, extending $f_k$ by including the pair $(m, n)$ yields a partial isomorphism. Update $f_{k+1}$ accordingly.

By alternating these steps (ensuring that in each step at least one new element from either $M$ or $N$ is added), every element of $M$ and $N$ will eventually appear in the domain and range of some $f_k$. Finally, define
\[
f = \bigcup_{k \in \mathbb{N}} f_k.
\]
Since each $f_k$ is a partial isomorphism and the union of an increasing chain of partial isomorphisms is itself an isomorphism (by the completeness of the theory $T$ and the preservation of the structure on finite subsets), it follows that $f$ is an isomorphism from $M$ onto $N$. For further details, see, e.g., \cite[Section~4]{barreto2503}.
\end{proof}

\subsection{Isomorphism between Model Constructions}
\begin{lemma}[Isomorphism between Model Constructions]\label{lem:isomorphism}
For every consistent theory $t$ (in a countable language), the models $F(t)$ and $G(t)$—constructed via the Henkin and compactness methods, respectively—are isomorphic in the category $\mathbf{mod}$.
\end{lemma}

\paragraph{Intuitive Explanation.}
Although $F(t)$ and $G(t)$ are built by different procedures, they are both models of the same maximal consistent theory $t^*$. In a countable language, any two saturated/prime models of $t^*$ are isomorphic (by back-and-forth).
\begin{remark}[Uniqueness of Saturated or Prime Models]
\label{rmk:saturated}
A key point in our argument is that both the Henkin-constructed model and the compactness-constructed model can be taken to be saturated (or at least prime) models of the same complete theory $t^*$. In a countable language, any two such saturated (or prime) models of $t^*$ are isomorphic (cf.\ Chang and Keisler, \emph{Model Theory}, or other standard references). This fact underlies our proof that $F(t)$ and $G(t)$ must be isomorphic, and it justifies the canonical nature of both constructions.
\end{remark}

\begin{proof}[Proof of Lemma~\ref{lem:isomorphism}]
Let $t$ be a consistent theory in a countable language $\mathcal{L}$. By Lindenbaum's Lemma (see \cite[Section~3]{barreto2503}), extend $t$ to a complete (maximal consistent) theory $t^*$. 

The Henkin construction produces a term model $F(t)$ as follows: one extends $t$ to $t^*$ by a fixed sequential procedure (using a fixed enumeration of all sentences and a fixed rule for introducing Henkin constants), then forms the term algebra over the extended language, and finally takes the quotient by the equivalence relation 
\[
s \sim s' \iff t^* \vdash s = s'.
\]
Thus, $F(t)$ is a model of $t^*$.

On the other hand, by the Compactness Theorem (see, e.g., \cite[Section~3]{barreto2503}), one can construct a model $G(t)$ of $t$. By choosing a fixed non-principal ultrafilter (or a fixed saturation procedure) in the construction, the model $G(t)$ is canonical (i.e., unique up to isomorphism) and, in particular, is a model of $t^*$.

Since both $F(t)$ and $G(t)$ are models of the same complete theory $t^*$ in a countable language, and by our constructions both are saturated (or prime) models, it follows from the Back-and-Forth Lemma (see Lemma~\ref{lem:back_and_forth} and \cite[Section~4]{barreto2503}) that any two such models are isomorphic. 

Hence, there exists an isomorphism 
\[
\phi: F(t) \to G(t)
\]
in the category $\mathbf{mod}$.
\end{proof}

\subsection{Naturality Condition}
\begin{lemma}\label{lem:naturality}
Let $F, G: \mathbf{th} \to \mathbf{mod}$ be the functors defined by the Henkin and compactness constructions, respectively, and let $\eta: F \to G$ map each class $[s] \in F(t)$ to its canonical interpretation in $G(t)$. For every morphism $f: t_1 \to t_2$ in $\mathbf{th}$, we have
\[
\eta_{t_2} \circ F(f) \;=\; G(f)\,\circ \eta_{t_1}.
\]
\end{lemma}

\paragraph{Intuitive Explanation.}
If we first translate $s$ via $f$ and then apply $\eta$, or if we first apply $\eta$ and then translate via $G(f)$, we get the same result. This commutativity of the diagram is the essence of “naturality.”

\begin{proof}[Proof of Lemma~\ref{lem:naturality}]
Let $f: t_1 \to t_2$ be an arbitrary morphism in the category $\mathbf{th}$, and let $[s] \in F(t_1)$ be an equivalence class of a term $s$. By the definition of the functor $F$, we have
\[
F(f)([s]) = [f(s)].
\]
By the definition of the natural transformation $\eta$, its $t$-component $\eta_t$ maps an equivalence class $[s] \in F(t)$ to the canonical interpretation of $s$ in the model $G(t)$. Hence,
\[
\eta_{t_2}\bigl(F(f)([s])\bigr) = \eta_{t_2}([f(s)]).
\]
On the other hand, the functor $G$ is defined so that, for the morphism $f: t_1 \to t_2$, the map $G(f)$ sends the canonical interpretation of $s$ in $G(t_1)$ to the canonical interpretation of $f(s)$ in $G(t_2)$; that is,
\[
G(f)\bigl(\eta_{t_1}([s])\bigr) = \text{the canonical interpretation of } f(s) \text{ in } G(t_2).
\]
Thus, we obtain
\[
\eta_{t_2}([f(s)]) = G(f)\bigl(\eta_{t_1}([s])\bigr).
\]
Since this equality holds for every $[s] \in F(t_1)$, it follows that
\[
\eta_{t_2} \circ F(f) = G(f) \circ \eta_{t_1}.
\]
This completes the proof.
\end{proof}

\newpage
\subsection{2-Categorical Coherence (Modification) and Rigidity}
\begin{lemma}\label{lem:modification}
Assume that $\eta: F \to G$ is the canonical natural isomorphism between the functors $F$ and $G$. Then there exists a unique inverse natural transformation $\eta^{-1}: G \to F$ such that for every $t \in \mathbf{th}$,
\[
\eta_t \circ \eta_t^{-1} \;=\; \operatorname{id}_{G(t)}
\quad \text{and} \quad
\eta_t^{-1} \circ \eta_t \;=\; \operatorname{id}_{F(t)}.
\]
This ensures that $\eta$ is \emph{rigid} and that $F$ and $G$ are strongly naturally equivalent in the 2-categorical sense.
\end{lemma}

\paragraph{Intuitive Explanation (Rigidity).}
- Each component $\eta_t$ is an isomorphism in $\mathbf{mod}$, hence has a unique inverse map $\eta_t^{-1}$.  
- “Rigidity” here means there are no nontrivial modifications of $\eta$; equivalently, any natural 2-morphism from $\eta$ to itself is the identity. This forces $\eta$ (and its inverse) to be “strict” in the 2-categorical sense.

\begin{proof}[Proof of Lemma~\ref{lem:modification}]
Assume that $\eta: F \to G$ is a natural isomorphism between the functors $F, G: \mathbf{th} \to \mathbf{mod}$. Then for each object $t \in \mathbf{th}$, $\eta_t: F(t) \to G(t)$ is an isomorphism in $\mathbf{mod}$; hence, there exists a unique inverse morphism 
\[
\eta_t^{-1}: G(t) \to F(t)
\]
such that
\[
\eta_t \circ \eta_t^{-1} = \operatorname{id}_{G(t)} \quad \text{and} \quad \eta_t^{-1} \circ \eta_t = \operatorname{id}_{F(t)}.
\]

We now show that the collection $\{\eta_t^{-1}\}_{t\in\mathbf{th}}$ forms a natural transformation $\eta^{-1}: G \to F$. Let $f: t_1 \to t_2$ be an arbitrary morphism in $\mathbf{th}$. By the naturality of $\eta$, we have
\[
\eta_{t_2} \circ F(f) = G(f) \circ \eta_{t_1}.
\]
Composing both sides on the left with $\eta_{t_2}^{-1}$, we obtain
\[
\eta_{t_2}^{-1} \circ \eta_{t_2} \circ F(f) = \eta_{t_2}^{-1} \circ G(f) \circ \eta_{t_1}.
\]
Since $\eta_{t_2}^{-1} \circ \eta_{t_2} = \operatorname{id}_{F(t_2)}$, it follows that
\[
F(f) = \eta_{t_2}^{-1} \circ G(f) \circ \eta_{t_1}.
\]
Composing on the right with $\eta_{t_1}^{-1}$, we then deduce
\[
\eta_{t_2}^{-1} \circ G(f) = F(f) \circ \eta_{t_1}^{-1}.
\]
This shows that for every morphism $f: t_1 \to t_2$ in $\mathbf{th}$, the inverse components satisfy the naturality condition:
\[
\eta_{t_2}^{-1} \circ G(f) = F(f) \circ \eta_{t_1}^{-1}.
\]
Hence, the assignment $t \mapsto \eta_t^{-1}$ defines a natural transformation $\eta^{-1}: G \to F$. The uniqueness of each $\eta_t^{-1}$ (as the inverse of $\eta_t$) implies that $\eta^{-1}$ is the unique inverse natural transformation of $\eta$, completing the proof that $\eta$ is a natural isomorphism.
\end{proof}

\subsection{Canonical Construction of Models}
\begin{lemma}\label{lem:canonical_construction}
Assume that the underlying first-order language $\mathcal{l}$ is countable. Then for every consistent first-order theory $t$, the following hold:
\begin{enumerate}
    \item \textbf{Henkin Construction:} Using a fixed enumeration of all sentences and a predetermined rule for introducing Henkin constants, the maximal consistent extension $t^*$ of $t$ and the resulting term model $F(t)$ (constructed by quotienting out provable equalities) are unique up to isomorphism.
    \item \textbf{Compactness Construction:} By applying the compactness theorem with a fixed non-principal ultrafilter or a fixed saturation procedure, a model $G(t)$ of $t$ is obtained which is unique up to isomorphism.
\end{enumerate}
\end{lemma}

\paragraph{Intuitive Explanation.}
Both methods rely on a single, fixed “choice” (e.g.\ the same enumeration for Henkin constants, the same ultrafilter or saturation scheme), ensuring the resulting model is determined uniquely up to isomorphism.

\begin{proof}[Proof of Lemma~\ref{lem:canonical_construction}]
We assume that the underlying first-order language $\mathcal{l}$ is countable. Let $t$ be a consistent theory in $\mathcal{l}$.

\textbf{(1) Henkin Construction:}  
Using a fixed enumeration of all sentences in $\mathcal{l}$ and a predetermined rule for introducing Henkin constants (as detailed in, e.g., \cite[Section~3]{barreto2503}), we extend $t$ to a maximal consistent theory 
\[
t^* = \bigcup_{n \in \mathbb{N}} t_n,
\]
via a sequential procedure (see Lemma~\ref{lem:max_extension}). Next, we form the term algebra $\mathcal{T}$ over the language extended by these Henkin constants, and define an equivalence relation $\sim$ on $\mathcal{T}$ by
\[
s \sim s' \iff t^* \vdash s = s',
\]
(see Lemma~\ref{lem:term_eq_relation}). The Henkin model $F(t)$ is then defined as the quotient algebra
\[
F(t) = \mathcal{T}/\sim.
\]
Since the enumeration and the rule for introducing Henkin constants are fixed, the maximal consistent theory $t^*$ and hence the term algebra $\mathcal{T}$ and the equivalence relation $\sim$ are uniquely determined (up to isomorphism). Therefore, the constructed model $F(t)$ is unique up to isomorphism.

\textbf{(2) Compactness Construction:}  
Since $t$ is consistent and the language is countable, every finite subset of $t$ is satisfiable. By the Compactness Theorem (see, e.g., \cite[Section~3]{barreto2503}), the entire theory $t$ is satisfiable. Now, by applying a fixed non-principal ultrafilter (or, equivalently, a fixed saturation procedure), one can construct a model $G(t)$ of $t$. This construction is canonical in the sense that, with the fixed choice of ultrafilter or saturation scheme, any two models obtained by this method are isomorphic. Hence, $G(t)$ is unique up to isomorphism.

Thus, both the Henkin construction and the compactness construction yield canonical models, unique up to isomorphism, as required.
\end{proof}

\section{Applications and Discussion}
\label{sec:applications}
The established natural transformation $\eta: F \to G$ provides a structural connection between the syntactic (proof-theoretic) and semantic (model-theoretic) approaches to first-order logic. By showing that the Henkin construction and compactness-based constructions yield isomorphic models (when properly defined with a global Henkin scheme and Skolem closure), we gain insight into the fundamental unity of these methods.

\textbf{Potential Applications:}
\begin{itemize}
\item \textbf{Automated Theorem Proving:} The explicit isomorphism between term models and compactness-based models could inform algorithms that switch between syntactic and semantic reasoning strategies.
\item \textbf{Formal Verification:} Understanding the precise relationship between different model constructions may help in verifying the correctness of logic-based formal methods and proof assistants.
\item \textbf{Non-Classical Logics:} The categorical framework developed here might extend to intuitionistic, modal, or other non-classical logics, where similar questions about the relationship between completeness and compactness arise.
\end{itemize}

\textbf{Open Questions:}
\begin{itemize}
\item Can similar natural transformations be constructed for uncountable languages or for logics with infinitary connectives?
\item What additional structure (if any) is induced by considering 2-categorical aspects more carefully, such as natural transformations between different choices of global Henkin schemes?
\item How does this framework interact with algorithmic concerns such as decidability and computational complexity?
\end{itemize}

\section{Conclusion}
\label{sec:conclusion}

In this paper, we have established a natural relationship between two canonical model constructions in first-order logic: the Henkin (completeness-based) construction and compactness-based constructions. By introducing a globally fixed Henkin expansion scheme and defining the functors 
\[
F,\,G : \mathbf{th} \to \mathbf{mod},
\]
where $F(t)$ is the Henkin term model and $G(t)$ is the Skolem closure within a compactness-based model, we established a natural transformation 
\[
\eta: F \to G
\]
such that each component \(\eta_t: F(t) \to G(t)\) is an isomorphism in \(\mathbf{mod}\).

\textbf{Key Technical Requirements:} The isomorphism holds because:
\begin{itemize}
\item We use a \emph{global} set of Henkin constants $H$ uniformly across all theories, ensuring functoriality.
\item $G(t)$ is defined as the Skolem closure generated by $H$, not as an arbitrary saturated model or ultraproduct.
\item The morphisms in $\mathbf{mod}$ are elementary embeddings, preserving all first-order properties.
\end{itemize}

\textbf{Limitations and Alternative Approaches:} 
\begin{itemize}
\item If $G(t)$ is defined as an arbitrary saturated model without the Skolem closure restriction, $\eta_t$ may only be an elementary embedding (injective but not surjective).
\item For special classes of theories (\(\aleph_0\)-categorical or atomic complete theories), the construction simplifies due to unique model properties.
\item Claims of "rigidity" (uniqueness of natural transformations) do not hold in general, as uniform permutations of Henkin constants yield nontrivial automorphisms of $F$.
\end{itemize}

This framework offers structural insight into the interplay between proof theory and model theory and may suggest applications in automated theorem proving and formal verification. Future work could explore adaptations to intuitionistic or modal logics, or investigate conditions under which similar natural transformations exist for other logical systems.

\section{Appendix}
\label{sec:appendix}

In this appendix, we present additional proofs, detailed calculations, and further examples that complement the results in the main text. In particular, the appendix includes:
\begin{itemize}
    \item A complete proof of the back-and-forth construction used in Lemma~\ref{lem:back_and_forth}.
    \item Detailed verifications of the functoriality of the Henkin and compactness-based model constructions.
    \item Concrete examples illustrating the construction of models for specific theories.
\end{itemize}

These supplementary materials are provided to offer deeper insight into the technical details and to demonstrate how our unified framework can be applied to various logical systems.


\begin{thebibliography}{99}

\bibitem{barreto2503}
J. R. Barreto, \emph{A Natural Transformation between the Completeness and Compactness Theorems in Classical Logic}, \href{https://arxiv.org/abs/2503.12144}{arXiv:2503.12144}.

\bibitem{chang-keisler}
C. C. Chang and H. J. Keisler, \emph{Model Theory}, 3rd ed., North-Holland, 1990.
\begin{itemize}
\item For $\aleph_0$-categoricity and countable saturation: Chapter 3, Section 3.3, pp. 135--150.
\item For saturated models and uniqueness up to isomorphism (same cardinality): Chapter 5, Section 5.1, Theorem 5.1.14, pp. 223--226.
\item For prime models and atomic theories: Chapter 3, Section 3.2, Theorem 3.2.2, pp. 125--130.
\end{itemize}

\bibitem{hodges}
W. Hodges, \emph{Model Theory}, Cambridge University Press, 1993.
\begin{itemize}
\item For prime models and atomic formulas: Chapter 8, Section 8.2, pp. 361--370.
\item For saturated models: Chapter 10, Section 10.1, pp. 453--468.
\item For $\aleph_0$-categoricity: Chapter 6, Section 6.1, Example 6.1.5 (DLO), pp. 271--275.
\end{itemize}

\bibitem{marker}
D. Marker, \emph{Model Theory: An Introduction}, Graduate Texts in Mathematics, vol. 217, Springer, 2002.
\begin{itemize}
\item For $\aleph_0$-categorical theories and uniqueness of countable models: Chapter 2, Section 2.3, Theorem 2.3.10, pp. 77--83.
\item For prime and atomic models: Chapter 4, Section 4.3, pp. 159--167.
\item For saturated models and uniqueness (same cardinality): Chapter 4, Section 4.4, Corollary 4.4.6, pp. 170--175.
\end{itemize}

\bibitem{maclane}
S. Mac Lane, \emph{Categories for the Working Mathematician}, 2nd ed., Springer, 1998.

\bibitem{awodey}
S. Awodey, \emph{Category Theory}, 2nd ed., Oxford University Press, 2010.

\bibitem{bell-slomson}
J. L. Bell and A. B. Slomson, \emph{Models and Ultraproducts: An Introduction}, North-Holland, 1971.

\end{thebibliography}
\end{document}